\documentclass[10pt]{article}
\usepackage[all,pdf]{xy}
\usepackage{amsmath}
\usepackage{amsfonts}
\usepackage{amssymb}
\usepackage{amsthm}
\usepackage{latexsym}
\usepackage{indentfirst}
\usepackage{bm}
\usepackage{graphicx}
\usepackage{epsfig}
\usepackage{mathrsfs}

\usepackage{xcolor}
\usepackage{bbm}
\usepackage{quiver}
\usepackage{url}
\usepackage{enumerate}

\usepackage{algorithm}
\usepackage{algorithmicx}
\usepackage{algpseudocode}
\usepackage{ulem}

\newtheorem{theorem}{Theorem}[section]
\newtheorem{example}[theorem]{Example}
\newtheorem{definition}[theorem]{Definition}
\newtheorem{proposition}[theorem]{Proposition}
\newtheorem{lemma}[theorem]{Lemma}
\newtheorem{corollary}[theorem]{Corollary}
\newtheorem{remark}[theorem]{Remark}

\usepackage{fancyhdr}

\usepackage{tikz-cd}

\usepackage{tikz}
\newcommand*{\circled}[1]{\lower.7ex\hbox{\tikz\draw (0pt, 0pt)%
		circle (.5em) node {\makebox[1em][c]{\small #1}};}}
\usepackage{geometry}
\newcommand{\qa}{kQ/I}

\newcommand{\lseq}[1]{L(g^{d(#1)-1}(#1))\cdots L(g(#1))L(#1)}
\newcommand{\llseq}[2]{L(g^{#2}(#1))\cdots L(g(#1))L(#1)}
\newcommand{\wwedge}{{^\wedge}}
\newcommand{\fbc}{f-BC~}
\newcommand{\fsbc}{$f_s$-BC~}
\newcommand{\fsbca}{$f_s$-BCA~}

\newcommand{\soc}{\mathrm{soc}}
\newcommand{\image}{\mathrm{Im}}
\newcommand{\sspan}{\mathrm{Span}}

\begin{document}
	
	\title{\bf \Large Trivial extensions of monomial algebras are symmetric fractional Brauer configuration algebras}
	\author{Yuming Liu$^a$ and Bohan Xing$^{a,*}$}
	\maketitle
	
	\renewcommand{\thefootnote}{\alph{footnote}}
	\setcounter{footnote}{-1} \footnote{\it{Mathematics Subject
			Classification(2020)}: 16G20, 16G10, 16S99.}
	\renewcommand{\thefootnote}{\alph{footnote}}
	\setcounter{footnote}{-1} \footnote{\it{Keywords}: Symmetric fractional Brauer configuration (algebra) of type S; Monomial algebra; Trivial extension; Admissible cut.}
	\setcounter{footnote}{-1} \footnote{$^a$School of Mathematical Sciences, Laboratory of Mathematics and Complex Systems, Beijing Normal University,
		Beijing 100875,  P. R. China. E-mail: ymliu@bnu.edu.cn (Y.M. Liu); bhxing@mail.bnu.edu.cn (B.H. Xing).}
	\setcounter{footnote}{-1} \footnote{$^*$Corresponding author.}
	
	{\noindent\small{\bf Abstract:} By providing equivalent definitions of fractional Brauer configuration algebras in certain special cases, we associate to each monomial algebra some combinatorial data called a fractional Brauer configuration, from which we construct a corresponding fractional Brauer configuration algebra. We show that this algebra is isomorphic to the trivial extension of the given monomial algebra. Furthermore, we establish a one-to-one correspondence between the isomorphism classes of monomial algebras and the equivalence classes of pairs consisting of a symmetric fractional Brauer configuration algebra of type~S with a free fractional-degree function and an admissible cut on it.}

	\section{Introduction}	

In the representation theory of algebras, monomial algebras form an important class of finite-dimensional quiver algebras. They admit a natural grading and are often used as testing grounds for various theories and conjectures. For example, the minimal two-sided projective resolution of a given monomial algebra can be described explicitly~\cite{Bar}, and the finitistic dimension conjecture is known to hold for monomial algebras~\cite{GKK}. Moreover, monomial algebras encompass algebras of diverse global dimensions and representation types, making them a rich and versatile class from the viewpoint of representation theory.

Symmetric algebras constitute another fundamental class of finite-dimensional algebras. Typical examples include group algebras of finite groups and trivial extensions of finite-dimensional algebras. The trivial extension construction appears frequently in representation theory, and various connections between trivial extension algebras and tilting theory have been established (see \cite{HW,TW,Ric1}). More recently, a characterization of trivial extension algebras in terms of quivers with relations was obtained in~\cite[Theorem~1.1]{FSTTV}.

In \cite{Sch}, Schroll shows that trivial extensions of gentle algebras, a special subclass of monomial algebras, are Brauer graph algebras, a class of symmetric special biserial algebras, and establishes a correspondence between gentle algebras and Brauer graph algebras with multiplicity one via admissible cuts. Green and Schroll~\cite{GS} later generalized these results, proving that trivial extensions of almost gentle algebras, which are monomial and properly contain gentle algebras, are Brauer configuration algebras, and obtained a similar correspondence using admissible cuts. More recently, an analogous correspondence between trivial extensions of skew-gentle algebras and skew-Brauer algebras, which also contain gentle algebras as a subclass, has been established in the non-monomial case, by Elsener, Guazzelli, and Valdivieso \cite{EGV} and by Soto \cite{Soto} independently. In our forthcoming work \cite{LLX}, we will show that their skew-Brauer graph algebras are in fact f-BCAs (see the next paragraph for an explanation of this notion). This can also be hinted at by our Example \ref{ex:skew-BGA}.

Very recently, Li and Liu gave a further generalization of Brauer configuration algebras in \cite{LL}. These algebras are defined using certain combinatorial data called fractional Brauer configurations (abbr. f-BCs) and are referred to as fractional Brauer configuration algebras (abbr. f-BCAs). In general, such algebras need not be symmetric, multiserial, or even finite-dimensional. However, if the underlying f-BC is of type~S, finite, and symmetric (see Section~2.2 for the precise definitions of these notions), then the corresponding f-BCA is symmetric and finite-dimensional. We call such an algebra a symmetric Brauer configuration algebra of type~S (abbr. symmetric $f_s$-BCA).

In this paper, we generalize the results of \cite{Sch} and \cite{GS}. For any monomial algebra, we construct a fractional Brauer configuration of type~S (abbr. $f_s$-BC) and define the corresponding fractional Brauer configuration algebra ($f_s$-BCA). We show that this algebra is isomorphic to the trivial extension of the given monomial algebra. Furthermore, we establish a one-to-one correspondence between the set of isomorphism classes of monomial algebras and the set of equivalence classes of pairs consisting of a symmetric $f_s$-BCA with a free fractional-degree function and an admissible cut on it.

\medskip
\textbf{Outline.} \;In Section~2, we recall the definitions and basic properties of f-BCAs and $f_s$-BCAs. We also provide equivalent conditions for the defining relations of f-BCAs in Lemma~\ref{relations}, and give an equivalent definition of $f_s$-BCAs in Lemma~\ref{f7'} in the case where the corresponding $f_s$-BC is symmetric. In Section~3, for each monomial algebra~$A$, we construct a fractional Brauer configuration $E_A$ of type~S (Proposition~\ref{EA is fsbc}). Section~4 shows that the $f_s$-BCA associated with $E_A$ is isomorphic to the trivial extension of the given monomial algebra~$A$ (Theorem~\ref{Trivial extension of mono}). In Section~5, we introduce admissible cuts of symmetric $f_s$-BCAs and establish a one-to-one correspondence between the isomorphism classes of monomial algebras and the equivalence classes of pairs consisting of a symmetric $f_s$-BCA with a free fractional degree function and an admissible cut on it (Corollary~\ref{1-1 cor}).

	\section{Basic knowledge about $f_s$-BCAs}

	\subsection{Fractional Brauer configuration algebras}
	
	We recall some definitions of fractional Brauer configuration algebras in \cite{LL}, which will play an important role in our following discussions. Furthermore, we give some equivalent conditions of the relations in fractional Brauer configuration algebras.
	
	\begin{definition}\textnormal{(\cite[Definition 3.3]{LL})}\label{f-BC}
		A fractional Brauer configuration (abbr. f-BC) is a quadruple $E=(E,P,L,d)$, where $E$ is a $G$-set with $G=\langle g\rangle\cong(\mathbb{Z},+)$, an infinite cyclic group, $P$ and $L$ are two partitions of $E$, and $d:E\rightarrow\mathbb{Z}_+$ is a function, such that the following conditions hold.
		\begin{enumerate}[(f1)]
			\item $L(e)\subseteq P(e)$ and $P(e)$ is a finite set for each $e\in E$.
			
			\item If $L(e_1)=L(e_2)$, then $P(g(e_1))=P(g(e_2))$.
			
			\item If $e_1,e_2$ belong to same $\langle g\rangle$-orbit, then $d(e_1)=d(e_2)$.
			
			\item $P(e_1)=P(e_2)$ if and only if $P(g^{d(e_1)}(e_1))=P(g^{d(e_2)}(e_2))$.
			
			\item $L(e_1)=L(e_2)$ if and only if $L(g^{d(e_1)}(e_1))=L(g^{d(e_2)}(e_2))$.
			
			\item The formal sequence $\lseq{e}$ is not a proper subsequence of the formal sequence $\lseq{h}$ for all $e,h\in E$.
		\end{enumerate}
	\end{definition}
	
	For convenience, we recall the following notations in \cite[Remark 3.4]{LL}. The elements in $E$ are called angles. The $\langle g\rangle$-orbits of $E$ are called vertices. The classes $P(e)$ of the partition $P$ are called polygons. The partition $L$ is said to be trivial if $L(e)=e$ for each $e\in E$. It will be clear later when we define the associated quiver of $E$ that the partitions $P$ and $L$ correspond to vertices and arrows respectively.
	
	The function $d:E\rightarrow\mathbb{Z}_+$ is called degree function. Condition (f3) means that the degree function can be defined on vertices. Let $v$ be a vertex such that $v$ is a finite set, define the fractional-degree (abbr. f-degree) $d_f(v)$ of a vertex $v$ to be the rational number $d(v)/|v|$. An f-BC $E$ is called f-degree-free (resp. the f-degree function of $E$ is free) if each $d(e)= |v_e|$ (respectively, $d_f(v_e)=1$) where $v_e$ is the $\langle g\rangle$-orbit containing the angle $e$.
	
	Denote by $\sigma$ the map $E\rightarrow E$, $e\mapsto g^{d(e)}(e)$, which is called the Nakayama automorphism of $E$. Moreover, we call an f-BC $E$ {\it symmetric} if its Nakayama automorphism is identity. This equals to say that $E$ has integral fractional-degree on each vertex.
	
	Below we provide some examples to help the reader understand these concepts.

	\begin{example}\label{ex:1-LBGA}
		Let $E=\{1,1',2,3,4,4'\}$. Define the group action on $E$ by $g(1)=2$, $g(2)=4$, $g(4)=1$, $g(1')=3$, $g(3)=4'$, $g(4')=1'$. Define $P(1)=\{1,1'\}$, $P(2)=\{2\}$, $P(3)=\{3\}$, $P(4)=\{4,4'\}$, $L(4)=\{4,4''\}$ and $L(e)=\{e\}$ for other $e\in E$. The f-degree function $d_f$ of $E$ is free. It can be verified that $E$ is an f-BC. 
		
		Since each polygon here contains at most two elements, we can represent each angle by a half-edge and each $2$-gon by an edge. The above combinatorial description can be illustrated by the following graph, where the orientation around each vertex corresponds to the group action, which we take to be clockwise. In particular, following the convention in~\cite{X}, we use a dotted line to represent an edge whose two half-edges lie in the same $L$-partition, indicating that this edge is labeled in the graph.

		\begin{center}
\tikzset{every picture/.style={line width=0.75pt}} 

\begin{tikzpicture}[x=0.75pt,y=0.75pt,yscale=-1,xscale=1]

\draw  [fill={rgb, 255:red, 0; green, 0; blue, 0 }  ,fill opacity=1 ] (100,115) .. controls (100,112.24) and (102.24,110) .. (105,110) .. controls (107.76,110) and (110,112.24) .. (110,115) .. controls (110,117.76) and (107.76,120) .. (105,120) .. controls (102.24,120) and (100,117.76) .. (100,115) -- cycle ;
\draw  [fill={rgb, 255:red, 0; green, 0; blue, 0 }  ,fill opacity=1 ] (200,115) .. controls (200,112.24) and (202.24,110) .. (205,110) .. controls (207.76,110) and (210,112.24) .. (210,115) .. controls (210,117.76) and (207.76,120) .. (205,120) .. controls (202.24,120) and (200,117.76) .. (200,115) -- cycle ;
\draw [line width=1.5]    (105,115) -- (205,115) ;
\draw [line width=1.5]  [dash pattern={on 5.63pt off 4.5pt}]  (105,115) .. controls (198,70.5) and (111,154.5) .. (205,115) ;
\draw [line width=1.5]    (105,115) -- (55,115) ;
\draw [line width=1.5]    (255,115) -- (205,115) ;

\draw (62,100) node [anchor=north west][inner sep=0.75pt]   [align=left] {$2$};
\draw (241,100) node [anchor=north west][inner sep=0.75pt]   [align=left] {$3$};
\draw (107,85) node [anchor=north west][inner sep=0.75pt]   [align=left] {$4$};
\draw (113,123) node [anchor=north west][inner sep=0.75pt]   [align=left] {$1$};
\draw (198,133) node [anchor=north west][inner sep=0.75pt]   [align=left] {$4'$};
\draw (184,93) node [anchor=north west][inner sep=0.75pt]   [align=left] {$1'$};
\end{tikzpicture}
		\end{center}
	\end{example}

	\begin{example}\label{ex:2-kx/x^3}
		Let $E=\{1,1',1''\}$. Define the group action on $E$ by $g(1)=1'$, $g(1')=1''$, $g(1'')=1$. Define $P(1)=\{1,1',1''\}$, $L(1)=\{1,1'\}$ and $L(1'')=\{1''\}$. The f-degree function $d_f$ of $E$ is free. It can be verified that $E$ is an f-BC.
		
		The above combinatorial data can be represented by the following diagram, in the same manner as in~\cite{GS2}, where the orientation around each vertex corresponds to the group action, which we take to be clockwise. In particular, the small arcs at the angles $1$ and $1'$ indicate that these two angles belong to the same class of the partition $L$, following the same convention as in~\cite[Example 7.16]{LL}.

		\begin{center}

\tikzset{every picture/.style={line width=0.75pt}} 

\begin{tikzpicture}[x=0.75pt,y=0.75pt,yscale=-1,xscale=1]

\draw  [fill={rgb, 255:red, 236; green, 236; blue, 236 }  ,fill opacity=1 ] (175,86) .. controls (175,44.58) and (208.58,11) .. (250,11) .. controls (291.42,11) and (325,44.58) .. (325,86) .. controls (325,127.42) and (291.42,161) .. (250,161) .. controls (208.58,161) and (175,127.42) .. (175,86) -- cycle ;
\draw [fill={rgb, 255:red, 255; green, 255; blue, 255 }  ,fill opacity=1 ]   (250,156) .. controls (249,10.5) and (141,98.5) .. (250,161) ;
\draw [fill={rgb, 255:red, 255; green, 255; blue, 255 }  ,fill opacity=1 ]   (250,161) .. controls (363,100.5) and (251,11.5) .. (250,156) ;
\draw  [fill={rgb, 255:red, 0; green, 0; blue, 0 }  ,fill opacity=1 ] (245,161) .. controls (245,158.24) and (247.24,156) .. (250,156) .. controls (252.76,156) and (255,158.24) .. (255,161) .. controls (255,163.76) and (252.76,166) .. (250,166) .. controls (247.24,166) and (245,163.76) .. (245,161) -- cycle ;
\draw  [draw opacity=0][line width=1.5]  (220.62,154.91) .. controls (221.5,150.63) and (223.3,146.68) .. (225.79,143.28) -- (250,161) -- cycle ; \draw  [line width=1.5]  (220.62,154.91) .. controls (221.5,150.63) and (223.3,146.68) .. (225.79,143.28) ;  
\draw  [draw opacity=0][line width=1.5]  (246.51,121.15) .. controls (247.66,121.05) and (248.83,121) .. (250,121) .. controls (251.17,121) and (252.34,121.05) .. (253.49,121.15) -- (250,161) -- cycle ; \draw  [line width=1.5]  (246.51,121.15) .. controls (247.66,121.05) and (248.83,121) .. (250,121) .. controls (251.17,121) and (252.34,121.05) .. (253.49,121.15) ;  

\draw (190,111) node [anchor=north west][inner sep=0.75pt]   [align=left] {$1$};
\draw (245,83) node [anchor=north west][inner sep=0.75pt]   [align=left] {$1'$};
\draw (301,111) node [anchor=north west][inner sep=0.75pt]   [align=left] {$1''$};

\end{tikzpicture}

		\end{center}
	\end{example}

	\begin{example}\label{ex:3-fBCA-without-admissible-cut}
		We present an example that is slightly more complicated than Example \ref{ex:2-kx/x^3}. Let $E=\{1,1',1'',1'''\}$. Define the group action on $E$ by $g(1)=1'$, $g(1')=1''$, $g(1'')=1'''$ and $g(1''')=1$. Define $P(1)=\{1,1',1'',1'''\}$, $L(1)=\{1,1'\}$ and $L(1'')=\{1'',1'''\}$. The f-degree function $d_f$ of $E$ is free. It can be verified that $E$ is an f-BC.
		
		The above combinatorial data can be represented by the following diagram, where the orientation around each vertex corresponds to the group action, which we take to be clockwise. In particular, the small arcs (resp. double-arcs) at the angles $1$ and $1'$ (resp. $1''$ and $1'''$) indicate that these two angles belong to the same class of the partition $L$.

		\begin{center}

\tikzset{every picture/.style={line width=0.75pt}} 

\begin{tikzpicture}[x=0.75pt,y=0.75pt,yscale=-1,xscale=1]

\draw  [fill={rgb, 255:red, 236; green, 236; blue, 236 }  ,fill opacity=1 ] (175,86) .. controls (175,44.58) and (208.58,11) .. (250,11) .. controls (291.42,11) and (325,44.58) .. (325,86) .. controls (325,127.42) and (291.42,161) .. (250,161) .. controls (208.58,161) and (175,127.42) .. (175,86) -- cycle ;
\draw [fill={rgb, 255:red, 255; green, 255; blue, 255 }  ,fill opacity=1 ]   (250,156) .. controls (226,93.5) and (163,120.5) .. (250,161) ;
\draw [fill={rgb, 255:red, 255; green, 255; blue, 255 }  ,fill opacity=1 ]   (250,161) .. controls (327,112.5) and (279,93.5) .. (250,156) ;
\draw [fill={rgb, 255:red, 255; green, 255; blue, 255 }  ,fill opacity=1 ]   (250,156) .. controls (283,58.5) and (225,57.5) .. (250,161) ;
\draw  [fill={rgb, 255:red, 0; green, 0; blue, 0 }  ,fill opacity=1 ] (245,161) .. controls (245,158.24) and (247.24,156) .. (250,156) .. controls (252.76,156) and (255,158.24) .. (255,161) .. controls (255,163.76) and (252.76,166) .. (250,166) .. controls (247.24,166) and (245,163.76) .. (245,161) -- cycle ;
\draw  [draw opacity=0][line width=1.5]  (220.76,154.25) .. controls (221.44,151.31) and (222.55,148.54) .. (224.01,146) -- (250,161) -- cycle ; \draw  [line width=1.5]  (220.76,154.25) .. controls (221.44,151.31) and (222.55,148.54) .. (224.01,146) ;  
\draw  [draw opacity=0][line width=1.5]  (236.85,134.03) .. controls (239.17,132.89) and (241.66,132.05) .. (244.27,131.55) -- (250,161) -- cycle ; \draw  [line width=1.5]  (236.85,134.03) .. controls (239.17,132.89) and (241.66,132.05) .. (244.27,131.55) ;  
\draw  [draw opacity=0][line width=1.5]  (256.24,131.65) .. controls (258.49,132.13) and (260.65,132.85) .. (262.68,133.8) -- (250,161) -- cycle ; \draw  [line width=1.5]  (256.24,131.65) .. controls (258.49,132.13) and (260.65,132.85) .. (262.68,133.8) ;  
\draw  [draw opacity=0][line width=1.5]  (273.96,142.94) .. controls (276.56,146.38) and (278.43,150.4) .. (279.35,154.76) -- (250,161) -- cycle ; \draw  [line width=1.5]  (273.96,142.94) .. controls (276.56,146.38) and (278.43,150.4) .. (279.35,154.76) ;  
\draw  [draw opacity=0][line width=1.5]  (257.88,126.89) .. controls (260.9,127.58) and (263.77,128.67) .. (266.43,130.09) -- (250,161) -- cycle ; \draw  [line width=1.5]  (257.88,126.89) .. controls (260.9,127.58) and (263.77,128.67) .. (266.43,130.09) ;  
\draw  [draw opacity=0][line width=1.5]  (277.75,139.67) .. controls (280.83,143.67) and (283.06,148.35) .. (284.18,153.45) -- (250,161) -- cycle ; \draw  [line width=1.5]  (277.75,139.67) .. controls (280.83,143.67) and (283.06,148.35) .. (284.18,153.45) ;  

\draw (190,116) node [anchor=north west][inner sep=0.75pt]   [align=left] {$1$};
\draw (226,99) node [anchor=north west][inner sep=0.75pt]   [align=left] {$1'$};
\draw (266,99) node [anchor=north west][inner sep=0.75pt]   [align=left] {$1''$};
\draw (297,114) node [anchor=north west][inner sep=0.75pt]   [align=left] {$1'''$};

\end{tikzpicture}

		\end{center}
	\end{example}

	In fact, the above combinatorial definition of an f-BC provides a way to define a quiver with relations, which in turn helps us study more (finite-dimensional) algebras. We will better understand the concepts in Definition~\ref{f-BC} through the quivers associated with f-BCs.
	
	We first recall some basic notations about quiver algebras. Let $k$ be a field and let $A=\qa$ be a quiver algebra, which means that $A$ is a quotient algebra of the path algebra $kQ$ defined on a finite quiver $Q$ by modulo an ideal $I$ in $kQ$. (Note that in the present paper all the involved ideals $I$ are lied in between the ideal $kQ_{\geq 1}$ generated by the arrows in $Q$ and some power of $kQ_{\geq 1}$.)  We denote by $s(p)$ the source vertex of a path $p$ and by $t(p)$ its terminus vertex. We will write paths from right to left, for example, $p=\alpha_{n}\alpha_{n-1}\cdots\alpha_{1}$ is a path with starting arrow $\alpha_{1}$ and ending arrow $\alpha_{n}$. 
	The length of a path $p$ will be denoted by $l(p)$. For convenience, for $p,q\in Q$, we denote $p\mid q$  if $p$ is a subpath of $q$.
	By abuse of notation we sometimes view an element in $kQ$ as an element in $\qa$ if no confusion can arise. 
	
	We now define, for each f-BC, a quiver with relations associated to it.
		
	\begin{definition}\textnormal{(\cite[Definition 4.1]{LL})}
	For an f-BC $E=(E,P,L,d)$, the quiver $Q_E=(Q_0,Q_1)$ associated with $E$ defined as follow: $Q_0=\{P(e)\mid e\in E\}$ and $$Q_1=\{L(e)\mid \text{$e\in E$, $s(L(e))=P(e)$ and $t(L(e))=P(g(e))$} \}.$$
		\end{definition}
	Therefore, the sequence $\lseq{e}$ we considered in condition $(f6)$ in Definition \ref{f-BC} is actually a path in the quiver $Q_E$. We call such a path of the form $\lseq{e}$ the {\it special path}  starting at $e\in E$. We also note that for a path $p=L(e_n)\cdots L(e_2)L(e_1)$ of $Q_E$, there exists some $e\in E$ such that $p=L(g^{n-1}\cdot e)\cdots L(g\cdot e)L(e)$ if and only if $\bigcap_{i=1}^n g^{n-i}\cdot L(e_i)\neq\emptyset$ (see \cite[Lemma 4.3]{LL}).
	
	Moreover, we can define the ideal $I_E$ which is generated by the following three types of relations (see \cite[Definition 4.4]{LL}):
	\begin{enumerate}[(R1)]
		\item $\llseq{e}{d(e)-1-k}-\llseq{h}{d(h)-1-k}$, if $P(e)=P(h)$, $k\geq 0$ and $L(g^{d(e)-i}(e))=L(g^{d(h)-i}(h))$ for $1\leq i\leq k$.
		
		\item $L(e_n)\cdots L(e_2)L(e_1)$, if $P(g(e_i))=P(e_{i+1})$ for each $1\leq i\leq n-1$ and $\bigcap_{i=1}^ng^{n-i}(L(e_i))=\emptyset$.
		
		\item $\llseq{e}{n-1}$ for $n>d(e)$.
	\end{enumerate}
	
	Call the three kinds of relations in $I_E$ type~1, type~2, and type~3, respectively.  The quiver algebra $A_E: = kQ_E / I_E$ is then called the fractional Brauer configuration algebra (abbr. f-BCA) associated with the f-BC $E$. Note that in~\cite[Definition 5.1]{LL}, the definition of the corresponding algebra is defined as the opposite algebra $kQ_E^{op} / I_E^{op}$, however, since it is more direct to connect $E$ with the quiver $Q_E$, we define $A_E$ as the present form.

	The above relations are expressed using the combinatorial notions appearing in the definition of f-BCs. For clarity, we restate their meanings below in the language of the path algebra $kQ_E$.	Consider the following types of relations in $kQ_E$.

\begin{enumerate}[(R1')]
	
	\item $$\llseq{e}{d(e)-1-k} = \llseq{h}{d(h)-1-k},$$ where $P(e) = P(h)$ and $k \ge 0$, provided that there exists a nonzero path $p$ in $Q_E$ such that 
	\[
	p\llseq{e}{d(e)-1-k} \quad \text{and} \quad p\llseq{h}{d(h)-1-k}
	\]
	are special paths starting at $P(e)$.
	
	\item $$L(e_n)\cdots L(e_2)L(e_1) = 0,$$ where the path $L(e_n)\cdots L(e_2)L(e_1)$ is not a subpath of any special path of an angle $e_1 \in E$, but every proper subpath of it is a subpath of some special path in $kQ_E$.
	
\end{enumerate}

	By Definition \ref{f-BC}, it is easy to see that the sequences in the \fbc $E$ fitting the condition (R1) is equivalent to the paths in $Q_E$ fitting the condition (R1'). Moreover, we give the following lemma to show two descriptions of the relations in the ideal in $kQ_E$ are equivalent.
	
	\begin{lemma}\label{relations}
		The ideal generated by (R2') in $kQ_E$ is equal to the ideal generated by (R2) and (R3).
	\end{lemma}
	
	\begin{proof}
		On the one hand, we prove each relation in (R2) and (R3) can be generated by (R2') in $kQ_E$. For all $\llseq{e}{n-1}$ which is a relation in (R3) with $n>d(e)$, it is obviously not a subpath of the special path $\lseq{e}$ of $e$, thus we can find some relation in (R2') that divides it exactly.
		If there exists a relation $L(e_n)\cdots L(e_2)L(e_1)$ which is a subpath of some special path in $kQ_E$, without loss of generality, we can assume it is in the form of $\llseq{e}{n-1}$. However, that means $e_n\in\bigcap_{i=1}^ng^{n-i}(L(e_i))\neq\emptyset$, contradicting with the condition in (R2).
		
		On the other hand, we prove each relation in (R2') can be generated by (R2) and (R3) in $kQ_E$. For each relation $L(e_n)\cdots L(e_2)L(e_1)$ in (R2'), if it does not contain $\lseq{e_1}$ as a proper subpath, consider the largest positive integer $m$ such that $L(e_m)\neq L(g^m(e_1))$, then by Definition \ref{f-BC}, in f-BC $E$, we have $L(e_{m})\cap L(g^{m}(e_1))=\emptyset$ since $L$ is a partition of $E$. Therefore, $L(e_m)\cdots L(e_2)L(e_1)$ is a relation in (R2) which is a subpath of $L(e_n)\cdots L(e_2)L(e_1)$.
		
		If $L(e_n)\cdots L(e_2)L(e_1)$ contains $\lseq{e_1}$ as a proper subpath, then there are two cases by considering the arrow $L(e_{d(e_1)})$ in $Q_E$.
		
		\begin{itemize}
			\item If $L(e_{d(e_1)})=L(g^{d(e_1)}(e_1))$, then it contains $L(g^{d(e_1)}(e_1))\lseq{e_1}$ which is a relation in (R3) as a subpath in $Q_E$.
			
			\item If $L(e_{d(e_1)})\neq L(g^{d(e_1)}(e_1))$, then by Definition \ref{f-BC}, in f-BC $E$, we have $L(e_{d(e_1)})\cap L(g^{d(e_1)}(e_1))=\emptyset$ since $L$ is a partition of $E$. Therefore, $L(e_{d(e_1)})\cdots L(e_2)L(e_1)$ is a relation in (R2) which is a subpath of $L(e_n)\cdots L(e_2)L(e_1)$.
		\end{itemize} 
		
		To sum up, the ideal generated by (R2') in $kQ_E$ is equal to the ideal generated by (R2) and (R3).
	\end{proof}

	\begin{remark}
		We note that, as a generating set of $I_E$, whether viewed from the combinatorial perspective (generated by relations (R1), (R2), and (R3)) or from the path algebra perspective (generated by relations (R1') and (R2')), it is not a minimal generating set of $I_E$. This is for the same reason that the three types of defining relations of the basic Brauer graph algebra do not form a minimal generating set. For instance, see~\cite[Remark~4.4]{LX}. In general, these relations constitute a Gr\"{o}bner basis of the ideal $I_E$. However, such a basis is often not reduced; that is, some relations are redundant and can be generated by others.
	\end{remark}

	We conclude this subsection by giving some examples.

	\begin{example}\label{ex-alg:1-LBGA}(Example \ref{ex:1-LBGA} revisited)
		 The quiver $Q_E$ is given by
\[\begin{tikzcd}
	& {P(1)} \\
	{P(2)} && {P(3)} \\
	& {P(4)}
	\arrow["{L(1)}"', from=1-2, to=2-1]
	\arrow["{L(1')}", from=1-2, to=2-3]
	\arrow["{L(2)}"', from=2-1, to=3-2]
	\arrow["{L(3)}", from=2-3, to=3-2]
	\arrow["{L(4)}"{description}, from=3-2, to=1-2]
\end{tikzcd}\]
		and for example, the special path starting at $1\in E$ is given by $$L(g^2(1))L(g(1))L(1)=L(4)L(2)L(1),$$ and the special path starting at $1'\in E$ is given by $$L(g^2(1'))L(g(1'))L(1')=L(4)L(3)L(1').$$
		
		The ideal $I_E$ is generated by 
		\begin{enumerate}[(R1')]
			\item $L(2)L(1)=L(3)L(1')$;
			
			\item \begin{itemize}
				\item $L(1)L(4)L(2)=L(1')L(4)L(3)=0$;
				\item $L(1)L(4)L(2)L(1)=L(4)L(2)L(1)L(4)=L(2)L(1)L(4)L(2)=0$;
				\item $L(1')L(4)L(3)L(1')=L(4)L(3)L(1')L(4)=L(3)L(1')L(4)L(3)=0$.
			\end{itemize}
		\end{enumerate}
		In fact, this algebra is not special multiserial, but it is derived equivalent to a Brauer configuration algebra (see, for instance, \cite[Example~2.14]{AZ}).
	\end{example}

		\begin{example}\label{ex-alg:2-kx/x^3}(Example \ref{ex:2-kx/x^3} revisited)
		 The quiver $Q_E$ is given by
$$
		\begin{tikzcd}
			P(1) \arrow["L(1)"', loop, distance=2em, in=215, out=145] \arrow["L(1'')"', loop, distance=2em, in=35, out=325]
		\end{tikzcd}$$
		and for example, the special path starting at $1\in E$ is given by $$L(g^2(1))L(g(1))L(1)=L(1'')L(1)^2,$$ and the special path starting at $1'\in E$ is given by $$L(g^2(1'))L(g(1'))L(1')=L(1)L(1'')L(1).$$
		
		The ideal $I_E$ is generated by 
		\begin{enumerate}[(R1')]
			\item $L(1'')L(1)=L(1)L(1'')$;
			
			\item \begin{itemize}
				\item $L(1'')L(1)^2L(1'')=L(1)L(1'')L(1)^2=L(1)^2L(1'')L(1)=0$;
				\item $L(1)^3=L(1')^2=0$.
			\end{itemize}
		\end{enumerate}
	\end{example}

		\begin{example}\label{ex-alg:3-fBCA-without-admissible-cut}(Example \ref{ex:3-fBCA-without-admissible-cut} revisited)
		 The quiver $Q_E$ is given by
$$
		\begin{tikzcd}
			P(1) \arrow["L(1)"', loop, distance=2em, in=215, out=145] \arrow["L(1'')"', loop, distance=2em, in=35, out=325]
		\end{tikzcd}$$
		and for example, the special path starting at $1\in E$ is given by $$L(g^3(1))L(g^2(1))L(g(1))L(1)=L(1'')^2L(1)^2,$$ and the special path starting at $1'\in E$ is given by $$L(g^3(1'))L(g^2(1'))L(g(1'))L(1')=L(1)L(1'')^2L(1).$$
		
		The ideal $I_E$ is generated by 
		\begin{enumerate}[(R1')]
			\item $L(1)^2L(1'')=L(1'')L(1)^2$, $L(1)L(1'')^2=L(1'')^2L(1)$;
			
			\item \begin{itemize}
				\item $L(1)L(1'')L(1)=L(1'')L(1)L(1'')=L(1)^3=L(1'')^3=0$;
				\item $L(1'')^2L(1)^2L(1'')=L(1'')L(1)^2L(1'')^2=L(1)^2L(1'')^2L(1)=L(1)L(1'')^2L(1)^2=0$.
			\end{itemize}
		\end{enumerate}
	\end{example}

	\subsection{Fractional Brauer configuration algebras of type S}\label{subsec:def-fsBGA}
	
	In this section, we recall a special class of f-BCAs, called the fractional Brauer configuration algebras of type~S in \cite{LL}. We begin by recalling some basic definitions.

	\begin{definition}\textnormal{(\cite[Definition 3.10]{LL})}
		Let $E$ be an f-BC, call a sequence $p=(g^{n-1}(e),\cdots,g(e),e)$ with $e\in E$ and $0\leq n\leq d(e)$ a standard sequence of $E$. In particular, we define $p=()_e$ when $n=0$ which is called a trivial sequence in $E$.
		
		A standard sequence of the form  $p=(g^{d(e)-1}(e),\cdots,g(e),e)$ with $e\in E$ is called a full sequence of $E$. Actually, there is a bijective map between the full sequences $p=(g^{d(e)-1}(e),\cdots,g(e),e)$ of $E$ and the special paths $\lseq{e}$ in $kQ_E$.
	\end{definition}
	
	For a standard sequence $p=(g^{n-1}(e),\cdots,g(e),e)$, we can define two associated standard sequences
	$$^\wedge p=\left\{
	\begin{array}{*{3}{lll}}
		(g^{d(e)-1}(e),\cdots,g^{n+1}(e),g^n(e))&,& \text{if $0<n<d(e)$;}\\
		()_{g^{d(e)}(e)}& ,& \text{if $n=d(e)$;}\\
		(g^{d(e)-1}(e),\cdots,g(e),e)&,&\text{if $n=0$ and $p=()_e$,}
	\end{array}
	\right.$$
	and
		$$p^\wedge =\left\{
	\begin{array}{*{3}{lll}}
		(g^{-1}(e),g^{-2}(e),\cdots,g^{n-d(e)}(e))&,& \text{if $0<n<d(e)$;}\\
		()_{e}& ,& \text{if $n=d(e)$;}\\
		(g^{-1}(e),\cdots,g^{-d(e)}(e))&,&\text{if $n=0$ and $p=()_e$.}
	\end{array}
	\right.$$
	Note that for a standard sequence $p$ of $E$, $^\wedge pp$ and $pp^\wedge$ are full sequences of $E$.

	For a standard sequence $p=(g^{n-1}(e),\cdots,g(e),e)$, define a formal sequence
		$$L(p) =\left\{
	\begin{array}{*{3}{lll}}
		\llseq{e}{n-1},\cdots,g^{n-d(e)}(e))&,& \text{if $0<n\leq d(e)$;}\\
		e_{P(e)}&,&\text{$p=()_e$.}
	\end{array}
	\right.$$ where $e_{P(e)}$ is the trivial path at vertex $P(e)$ in $kQ_E$. Actually, it is a subpath of the special path of $e\in E$. Moreover, for a set $\mathcal{X}$ of standard sequences, define $L(\mathcal{X})=\{L(p)\mid p\in\mathcal{X}\}$. By the definition of f-BCA, for any standard sequence $p$ in f-BC $E$, the formal sequence $L(p)$ corresponds to a nonzero path in the associated algebra $A=kQ_E/I_E$. 
	By abuse of notation we sometimes view a standard sequence $p$ in $E$ as a path in $kQ_E$ if no confusion can arise. 
	
	\begin{definition}\textnormal{(\cite[Definition 3.11]{LL})}
		Let $E$ be an f-BC, $p$ and $q$ be two standard sequences of $E$, define $p\equiv q$ if $L(p)=L(q)$. In the case we say $p$ and $q$ are identical.
	\end{definition}
	
	Actually, for standard sequences $p,q$, $^\wedge p\equiv ^\wedge q$ if and only if $p^\wedge\equiv q^\wedge$. For a set $\mathcal{X}$ of standard sequences, denote $[\mathcal{X}]=\{\text{standard sequence $q$}\mid \text{$q$ is identical to some $p\in\mathcal{X}$}\}$, denote $^\wedge\mathcal{X}=\{^\wedge p\mid p\in\mathcal{X}\}$ (resp. $\mathcal{X}^\wedge=\{p^\wedge\mid p\in\mathcal{X}\}$).
	
	\begin{definition}\textnormal{(\cite[Definition 3.13]{LL})}
		An f-BC $E$ is said to be of type S (or $E$ is an $f_s$-BC for short) if it satisfies additionally the following condition.
		\begin{enumerate}[(f7)]
			\item For standard sequences $p\equiv q$, we have $\left[\left[{^\wedge} p\right]{^\wedge}\right]=\left[\left[{^\wedge}q\right]{^\wedge}\right]$.
		\end{enumerate}
		The algebra $A=kQ_E/I_E$ is called the fractional Brauer configuration algebra of type S (abbr. $f_s$-BCA) associated with an $f_s$-BC $E$.
	\end{definition}
	
	\begin{remark}
		For a standard sequence $p$, we have $\left[{^\wedge} \left[p{^\wedge}\right]\right]=\left[{^\wedge}\left[({^\wedge} p){^\wedge}{^\wedge}\right]\right]=\left[{^\wedge}(\left[{^\wedge} p\right]{^\wedge}{^\wedge})\right]=\left[\left[{^\wedge} p\right]{^\wedge}\right]$ whenever $(f7)$ holds or not.
	\end{remark}

	As the definition of an $f_s$-BC is formulated using the notation of an f-BC $E$, we restate it in terms of paths in $kQ_E$.

	From now on, assume that the f-BC $E$ is symmetric, meaning that its Nakayama automorphism is the identity. In this case, all special paths in $Q_E$ are cycles, and $\wwedge p = p\wwedge$ holds for every standard sequence $p$ in $E$. We now introduce the following additional condition.
	\begin{enumerate}[\textit{(sf7)}]
		\item For two nonzero relations $p-q$, $p'-q'$ of type 1 in $I_E$, if $pp'$, $qp'$ and $pq'$ are some special paths in $Q_E$ at the same time, then so is $qq'$.
	\end{enumerate}
	The following lemma shows that conditions $(f7)$ and $(sf7)$ are equivalent for symmetric f-BCs.
	
	\begin{lemma}\label{f7'}
		The condition (f7) implies (sf7). Moreover, if the f-BC $E$ is symmetric, then the condition (sf7) also implies (f7).
	\end{lemma}
	
	\begin{proof}
		On the one hand, we show that $(f7)$ implies $(sf7)$. Suppose, to the contrary, that $(sf7)$ does not hold. Then there exist two nonzero relations of type~1, say $p - q$ and $p' - q'$, in $I_E$ such that $pp'$, $qp'$, and $pq'$ are distinct special paths in $Q_E$, while $qq'$ is not. Denote $l(p') = n$. By condition~(f6) in Definition~\ref{f-BC}, we have $n \ge 1$. In this case, we can find distinct angles $e, h \in E$ such that 
\[
p' = \llseq{e}{n-1} = \llseq{h}{n-1},
\]
and 
\[
p = L(g^{d(e)-1}(e)) \cdots L(g^{n}(e)), \quad 
q = L(g^{d(h)-1}(h)) \cdots L(g^{n}(h)).
\]

		Denote the standard sequences corresponding to $p'$ by $$p_1=(g^{n-1}(e),\cdots,e)$$ and $$p_2=(g^{n-1}(h),\cdots,h),$$ thus $p_1\equiv p_2$. Since $pq'$ is a special path in $E$ but $qq'$ is not, we have $q'\in L(\left[\left[{^\wedge} p_1\right]{^\wedge}\right])$ but $q'\notin L(\left[\left[{^\wedge} p_2\right]{^\wedge}\right])$. Therefore, $(f7)$ is also not true at the same time.
		
		On the other hand, we prove $(sf7)$ implies $(f7)$ when $E$ is symmetric. If not, there exist standard sequences $p\equiv q$ and $p\neq q$, such that $\left[\left[{^\wedge} p\right]{^\wedge}\right]\neq\left[\left[{^\wedge}q\right]{^\wedge}\right]$, which also means $L(\left[{^\wedge} p\right]{^\wedge})\neq L(\left[{^\wedge}q\right]{^\wedge})$.
		
		To be more specific, we may assume that there exist a path $p_0$ in $Q_E$, such that $L(\wwedge p)p_0$ is a special path in $Q_E$, but $L(\wwedge q)p_0$ is not. In this case, we have $L(\wwedge p)\neq L(\wwedge q)$. Actually, by the definition of standard sequences and $L(p)=L(q)$, we have $L(\wwedge p)L(p)$ and $L(\wwedge q)L(p)$ are special paths in $Q_E$. Since $E$ is symmetric, all special paths are cycles in $Q_E$. Therefore, $L(p)L(\wwedge p)$ and $L(p)L(\wwedge q)$ are different special paths in $Q_E$ at the same vertex, that means $L(\wwedge p)-L(\wwedge q)$ is a relation of type 1. In the same reason, we have $L(p)-p_0$ is also a relation of type 1. However, $L(\wwedge p)-L(\wwedge q)$ and $L(p)-p_0$ do not fit the condition $(sf7)$, a contradiction!
	\end{proof}

	By Lemma~\ref{f7'}, the f-BCs in Examples~\ref{ex:1-LBGA}--\ref{ex:3-fBCA-without-admissible-cut} are all $f_s$-BCs. We refer the reader to \cite[Examples~3.7 and~3.8]{LL} for examples of f-BCs that are not $f_s$-BCs.
	
	\begin{proposition}\textnormal{(\cite[Proposition 5.2 and 5.4]{LL})}\label{Frob}
		If $E$ is a \fsbc with a finite angle set, then the corresponding \fsbca $A=kQ_E/I_E$ is a finite-dimensional Frobenius algebra with the Nakayama automorphism of $A$ induced by the inverse of the Nakayama automorphism of the \fsbc $E$.
	\end{proposition}
	
	Therefore, if the \fsbc with a finite angle set is symmetric, then by \cite[Proposition 5.5]{LL}, $A = kQ_E / I_E$ is also symmetric; that is, $A \cong \mathrm{Hom}_k(A, k)$ as $A$--$A$-bimodules. (See for example, \cite[Theorem~3.1]{Ric} for equivalent characterizations of symmetric algebras). Such algebras are called symmetric fractional Brauer configuration algebras of type~S (abbr. symmetric $f_s$-BCAs).

	\section{The $f_s$-BC associated to a monomial algebra}
	
	In this section, we construct an \fsbc from a given monomial algebra. Actually, it is a generalization of the graph of a gentle algebra in \cite[Section 3.1]{Sch}.

	Let $A = kQ / I$ be a finite-dimensional monomial algebra, which means that $I$ is an ideal in $kQ$ generated by paths. 
	Consider the set $\mathcal{M} = \{p_1, \dots, p_m\}$ of maximal paths in $A$, that is, for each $p \in \mathcal{M}$ and each arrow $\alpha \in Q_1$, we have $\alpha p = 0 = p \alpha$ in $A$. 
	The set $\mathcal{M}$ is uniquely determined by $A = kQ / I$, since by \cite[Proposition~2.5]{Green}, the ideal $I$ admits a unique minimal generating set $\mathcal{G}$ consisting of paths. 
	Consequently, the set of paths in $kQ$ that contain no subpath from $\mathcal{G}$ forms a finite $k$-linear basis $\mathcal{B}$ of $A$. 
	The subset $\mathcal{M} \subseteq \mathcal{B}$ consisting of maximal elements is therefore unique for a given monomial algebra $A = kQ / I$.

	Define a quadruple $E_A=(E,P,L,d)$ of the monomial algebra $A=\qa$ as follows.
	
	\begin{itemize}
		\item $E=\bigcup_{p\in\mathcal{M}}\{(e_i,p)\;|\;p=(e_1\rightarrow e_2\rightarrow\cdots\rightarrow e_n)\}$.
		
		\item $P((e_i,p))=\{(e_i',p')\in E\mid e_i=e_i'\in Q_0\}$.
		
		\item $L((e_i,p))=\{(e_i',p')\;|\;\text{the arrow starting at $e_i$ in $p$ is same as the arrow starting at $e_i'$ in $p'$}\}$. In particular, for every \(p \in \mathcal{M}\), we set \(L((t(p), p)) = \{(t(p), p)\}\).
		
		\item $d((e_i,p))=l(p)+1$.
		
		\item if $p=(e_1\rightarrow e_2\rightarrow\cdots\rightarrow e_n)$, then $g((e_i,p))=(e_{i+1},p),i=1,\cdots,n-1$ and $g((e_n,p))=(e_{1},p)$.
	\end{itemize}
	We emphasize that we regard $(e_i, p)$ and $(e_j, p)$ with $i \neq j$ as distinct angles in $E_A$, even when $e_i = e_j$ in $Q_0$. 
	It is easy to prove that $E_A$ is f-degree-free, that $P$ and $L$ form partitions of $E_A$, and that each $\langle g \rangle$-orbit corresponds to a unique maximal path in $A$.

	Now we show that the above definition indeed defines an $f_s$-BC.

	\begin{proposition}\label{EA is fsbc}
		$E_A$ is a symmetric $f_s$-BC.
	\end{proposition}
	
	\begin{proof}
		We check the conditions in Definition \ref{f-BC} step by step.
		
		\textit{(f1)}.\quad Since $A$ is finite-dimensional and $\mathcal{M}$ is contained in a $k$-basis of $A$, by the definition of $E_A$, the angle set $E$ is finite. Moreover, each $P((e_i,p))\subseteq E$ is finite. By the definition of the partitions $P$ and $L$ in $E_A$, for all $(e_i',p')\in L((e_i,p))$, there exists an common subarrow $\alpha\in Q_1$ of $p$ and $p'$, such that $s(\alpha)=e_i=e_i'$, thus $(e_i',p')\in P((e_i,p))$, that means $L((e_i,p))\subseteq P((e_i,p))$ for all $(e_i,p)\in E$.
		
		\textit{(f2)}.\quad If $L((e_i,p))=L((e_i',p'))$, then there exists an arrow $\alpha$ in $Q$ with $\alpha\mid p$ and $\alpha\mid p'$, such that $s(\alpha)=e_i=e_i'\in Q_0$. Denote the terminus vertex of $\alpha$ in $p$ and $p'$ by $e_j$ and $e_j'$, respectively. Then $e_j=e_j'\in Q_0$. Therefore, 
		$$P(g(e_i,p))=P((e_j,p))=P((e_j',p'))=P(g(e_i',p')).$$
		
		\textit{(f3)}.\quad For all $(e_i,p)\in E$, each $\langle g\rangle$-orbit of $(e_i,p)$ is defined by the maximal path $p\in\mathcal{M}$. Therefore, by the definition of the degree function, we have $d((e_i,p))=d(g^k(e_i,p))$, $k\in\mathbb{Z}$.
		
		\textit{(f4)} and \textit{(f5)}.\quad For all $e\in E$ which is an angle of $E_A$, we have $g^{d(e)}(e)=e$. Therefore, the conditions \textit{(f4)} and \textit{(f5)} are automatically established.
		
		\textit{(f6)}.\quad If $\lseq{e}$ is a proper subsequence of $\lseq{h}$ for some $e,h \in E$, without loss of generality, we can assume that there exist a positive integer $n<d(h)$, such that $$\lseq{e}=\llseq{h}{n-1}.$$ To be more specific, let $e=(e_i,p)$ and $h=(e_j,q)$. By the definition of the partition $L$ of $E_A$, there exist a non-trivial path $p'$ in $Q$, such that $q=p'p$. However, that means $p$ is a proper subpath of $q$, contradict to $p\in\mathcal{M}$ which is a maximal path in $A$.
		
		To sum up, by Definition \ref{f-BC}, we have $E_A$ is a \fbc. Moreover, since for all $e\in E$ which is an angle of $E_A$, we have $g^{d(e)}(e)=e$. Therefore,  the Nakayama automorphism $\sigma$ of $E_A$ is identity. That means, \fbc $E_A$ is symmetric.
		
		By using Lemma \ref{f7'}, we show that the \fbc $E_A$ fits the condition \textit{(sf7)}. If not, consider the quiver $Q_{E_A}$ associated with $E_A$, there exist $p_i:=\llseq{h_i}{n_i-1}$ with $i=1,2,3,4$, and two nonzero paths $q_1,q_2$ in $Q_{E_A}$ such that $q_1p_1$, $q_1p_2$, $q_2p_3$, $q_2p_4$ are special paths in $Q_{E_A}$. In other words, $p_1-p_2$, $p_3-p_4$ are relations of type 1 in $E_A$. Moreover, we can assume that $p_1p_3$, $p_2p_3$, $p_1p_4$ are special paths in $E_A$, but $p_2p_4$ is not.
		
		However, since $p_1\neq p_2$ and $q_1p_1$, $q_1p_2$ are special paths in $Q_E$, we have the corresponding elements $M_1$, $M_2$ in $\mathcal{M}$. Since $L((t(M_1),M_1))$ is trivial, its corresponding arrow in $Q_{E_A}$ can only appear in exactly one special path (under cyclic permutation of cycles in $Q_{E_A}$) in $Q_{E_A}$. Therefore, $L((t(M_1),M_1))\mid p_1$. Moreover, since $p_3\neq p_4$ and $p_1p_3$, $p_1p_4$ are special paths in $Q_E$, without loss of generality, we have $p_4\neq q_1$ (otherwise, $p_2p_4=p_2q_1$ is naturally a special path). Therefore, the arrow $L((t(M_1),M_1))$ appears in distinct special paths $q_1p_1$ and $p_1p_4$ in $Q_{E_A}$. However, $L((t(M_1),M_1))$ is a trivial angle set, which can only be involved in exactly one $\langle g\rangle$-orbit, a contradiction!
		
		In conclusion, $E_A$ fits \textit{(sf7)}. By Lemma \ref{f7'}, $E_A$ fits \textit{(f7)} since $E_A$ is symmetric. Therefore, $E_A$ is a symmetric \fsbc.
	\end{proof}

	Let $A = kQ / I$ be a finite-dimensional monomial algebra, and let $E_A = (E, P, L, m)$ be the \fsbc\ associated with $A$. 
	Denote by $A_E$ the $f_s$-BCA corresponding to $E_A$. 
	Since the angle set $E$ is finite, $A_E$ is finite-dimensional. 
	By Proposition~\ref{Frob} and the subsequent discussion, $A_E$ is a finite-dimensional symmetric algebra.

	\section{Trivial extensions of monomial algebras}

	Let $A = kQ / I$ be a finite-dimensional $k$-algebra, and let $D(A) = \mathrm{Hom}_k(A, k)$ denote its $k$-linear dual. Recall that the trivial extension of $A$, denoted by $T(A) = A \rtimes D(A)$, is the algebra whose underlying vector space is $A \oplus D(A)$ and whose multiplication is given by
\[
(a, f)(b, g) = (ab,\, ag + fb),
\]
for all $a, b \in A$ and $f, g \in D(A)$. Here $D(A)$ is viewed as an $A$--$A$-bimodule via the action defined as follows: for $a, b \in A$ and $f \in D(A)$,
\[
(afb)(x) = f(bxa), \quad \text{for all } x \in A.
\]
It is well known that the trivial extension algebra $T(A)$ is symmetric (see, for example, \cite[Proposition~6.5]{RS}). Moreover, it is shown in \cite[Proposition~2.2]{FP} that the vertices of the quiver $Q_{T(A)}$ of $T(A)$ correspond to those of the quiver $Q$ of $A$, and that the number of arrows from a vertex $i$ to a vertex $j$ in $Q_{T(A)}$ equals the number of arrows from $i$ to $j$ in $Q$, plus the dimension of the $k$-vector space $e_i(\mathrm{soc}_{A^e} A)e_j$.

Suppose that $I$ is generated by paths, that is, $A$ is a monomial algebra. Consider 
\[
\mathcal{B} = \{\, p \in Q \mid p \notin I \,\}.
\]
The set $\pi(\mathcal{B})$, where $\pi : kQ \to A$ is the canonical surjection, forms a $k$-basis of $A$. By abuse of notation, we identify $\mathcal{B}$ itself with this $k$-basis. By \cite[Proposition~2.2]{FP}, the set $\mathcal{M}$ of maximal paths of $A$ is a subset of $\mathcal{B}$ and forms a $k$-basis of $\soc_{A^e}(A)$. Hence, if we denote by $(Q_{T(A)})_1$ the arrow set of the quiver $Q_{T(A)}$ of $T(A)$, and by $Q_1$ the arrow set of the quiver $Q$ of the monomial algebra $A = kQ/I$, then  
\[
|(Q_{T(A)})_1| = |Q_1| + |\mathcal{M}|.
\]

	\begin{lemma}\label{quiver of AE/TA}
		Let $A=\qa$ be a finite-dimensional monomial algebra. Denote the \fsbca associated with $A$ by $A_E=kQ_E/I_E$ and the trivial extension of $A$ by $T(A)=kQ_{T(A)}/I_{T(A)}$. Then the quiver $Q_E$ is isomorphic to $Q_{T(A)}$.
	\end{lemma}
	
	\begin{proof}
		Denote the \fsbc associated with $A$ by $E_A=(E,P,L,m)$. Since the vertices of $Q_E$ are corresponding to the partition $P$ of $E$ and for all $(e_i,p)\in E$, the angles in $P((e_i,p))$ have a common first coordinate in $Q_0$, the vertices in $Q_E$ are corresponding to the vertices in $Q$. Thus the vertices in $Q_E$ are corresponding to the vertices in $Q_{T(A)}$.
		
		Now consider the arrows in $Q_E$ from a vertex $i$ to a vertex $j$ in $Q_E$. Denote the trivial path corresponding to $i$ and $j$ by $e_i$ and $e_j$ in $Q$. For each arrow $\alpha$ in $e_jQe_i$, it can be extended to a maximal path $p$ (may not unique) in $A$. Moreover, by definition of the \fsbc $E_A$, $g(e_i,p)=(e_j,p)$. Therefore, by definition of the quiver of \fbc, there is an arrow in $Q_E$ corresponding to $\alpha$. If we choose a different maximal path $p'$ containing $\alpha$, then by definition of the \fsbc $E_A$, $L((e_i,p))=L((e_i,p'))$. Thus this correspondence is a bijection between the arrows in $Q_{T(A)}$ corresponding to an arrow $\alpha\in e_jQe_i$ in $Q$ and the arrows $L((e_i,p))$ in $Q_E$ with $p$ a maximal path in $A$ containing $\alpha$. Moreover, for all maximal path $p\in\mathcal{M}$ from $j$ to $i$ (Note that $p\in e_i(\soc_{A^e}A)e_j$), we have that $L((e_i,p))$ is trivial and $g(e_i,p)=(e_j,p)$. Therefore, there exist a unique arrow $\alpha_p:=L((e_i,p))$ in $Q_E$ from $i$ to $j$ corresponding to the maximal path $p$.
		
		In conclusion, the quiver $Q_E$ is isomorphic to $Q_{T(A)}$.
	\end{proof}

	Recall some basic properties of $T(A)$. The dual basis $\mathcal{B}^\vee = \{\, p^\vee \mid p \in \mathcal{B} \,\}$ forms a $k$-basis of $D(A)$, where for each $p \in \mathcal{B}$, the element $p^\vee \in D(A)$ is defined by $p^\vee(q) = \delta_{p,q}$ for all $q \in \mathcal{B}$.
	
	\begin{lemma}\textnormal{(\cite[Lemma 4.1]{GS})}\label{opera}
		Let $A$ be a finite-dimensional monomial algebra with $k$-basis $\mathcal{B}$ as above. Then, for $p,q,r\in\mathcal{B}$, the following holds in $T(A)$.
		\begin{enumerate}[(1)]
			\item $(p,0)(0,r^\vee)=\left\{
			\begin{array}{*{3}{lll}}
				(0,s^\vee)&,& \text{if there is some $s\in\mathcal{B}$ with $sp=r$}\\
				0& ,& \text{otherwise.}
			\end{array}
			\right.$
			
			\item $(0,r^\vee)(q,0)=\left\{
			\begin{array}{*{3}{lll}}
				(0,s^\vee)&,& \text{if there is some $s\in\mathcal{B}$ with $qs=r$}\\
				0& ,& \text{otherwise.}
			\end{array}
			\right.$
			
			\item $(0,p^\vee)(q,0)(0,r^\vee)=0.$
			
			\item If $prq\in\mathcal{B}$, then $(q,0)(0,(prq)^\vee)(p,0)=(0,r^\vee)$.
		\end{enumerate}
	\end{lemma}
	
	\begin{proposition}\textnormal{(\cite[Proposition 4.2]{GS})}\label{basis of TA}
		Let $A$ be a finite-dimensional monomial algebra. Then $T(A)$ is generated by $\{(\alpha,0)\mid\alpha\in Q_1\}\cup\{(0,m^\vee)\mid m\in\mathcal{M}\}$.
	\end{proposition}
	
	We now prove the main result of this section.
	
	\begin{theorem}\label{Trivial extension of mono}
		Let $A=\qa$ be a finite-dimensional monomial algebra, $A_E$ the symmetric $f_s$-BCA of $E_A$, and $T(A)$ the trivial extension of $A$ by $D(A)$. Then $A_E$ is isomorphic to $T(A)$.
	\end{theorem}
	
	\begin{proof}
		By Lemma \ref{quiver of AE/TA}, we can divide arrows in $Q_E$ into two parts. Denote the arrow in $ Q_E$ corresponding to some arrow $\alpha$ in $Q$ by $\alpha$, and the arrow in $Q_E$ corresponding to some arrow induced by some $m\in\mathcal{M}$ by $\alpha_m$.
		
		We first prove $\dim_k A_E=2\dim_k A$. Actually, by definition of the \fsbc $E_A$ and the quiver associated with $E_A$, we can regard the monomial algebra $A=\qa$ as a subalgebra of $A_E$. Therefore, we can embedding the $k$-basis $\mathcal{B}$ of $A$ to nonzero paths in $A_E$ without involving arrows induced by maximal paths in $A$. To be more specific, it is an injection $i_1: A\rightarrow A_E$ of $k$-vector space given by $p\mapsto p$ for all nonzero path $p\in\mathcal{B}$ in $A$. 
		
		Recall each nonzero path in $A_E$ is induced by a standard sequence in $E_A$. 
		Thus, we can define the second map $i_2: A\rightarrow A_E$ of $k$-vector space which is given by $p\mapsto\ ^\wedge p$ (see \cite[Lemma 4.18]{LL} for the notation). By \cite[Lemma 4.18]{LL}, this is also an injection. Moreover, all nonzero paths in $\image i_2$ have a subpath which is an arrow that induced by maximal paths in $A$. Therefore, $\image i_1\cap \image i_2=\emptyset$. Moreover, we prove $\sspan_k(\image i_1\cup \image i_2)=A_E$. It is obviously to find that $\sspan_k(\image i_1\cup \image i_2)\subseteq A_E$. For all non zero path $q$ in $A$, if $q$ does not contain a subarrow induced by some maximal path in $A$, then $q\in\image i_1$. If there exist an arrow induced by some maximal path in $A$ that is a subpath of $q$, then by definition of \fsbca $A_E$, $\wwedge q\in\image i_1$. Therefore, $q=\wwedge\wwedge q\in\image i_2$. To sum up, we have  $\dim_k A_E=2\dim_k A=\dim_k T(A)$.
		 
		We construct a surjection $\psi: kQ_E\rightarrow T(A)$ of $k$-algebras which is given by $\alpha\in Q_1\mapsto(\alpha,0)$, $\alpha_m\mapsto (0,m^\vee)$ with $m\in\mathcal{M}$. It is straightforward to see that it is a surjection by Proposition \ref{basis of TA}. Now we prove it can induce a surjection from $A_E$ to $T(A)$, which means for all relations $\rho\in I_E$, $\psi(\rho)=0$.	
		
		By discussion in Lemma \ref{relations}, $I_E$ can be generated by relations fitting the condition (R1') or the condition (R2').
		
		Let $\rho=\alpha_{n}\cdots\alpha_{1}$ is a relation fitting the condition (R2'). There are two cases to consider as follows.
		
		\textit{Case 1.} If $\rho$ does not contain arrows induced by maximal paths in $A$, then by definition of $E_A$, it is actually a relation in the ideal $I$ in $A$. Then $\psi(\rho)=(\rho,0)=0$ in $T(A)$.
		
		\textit{Case 2.} If $\rho$ contains an arrow $\alpha_i$ induced by a maximal path $m\in\mathcal{M}$ in $A$, then we can write $m=\beta_n\cdots\beta_1$. If $\rho$ have some subarrow which is not in $\{\beta_1,\cdots,\beta_n\}$, by Lemma \ref{opera}, $\psi(\rho)=0$. If all arrows in $\rho$ is in $\{\beta_1,\cdots,\beta_n\}$, then  without loss of generality, we can assume $\rho=\beta_n\cdots\beta_1\alpha_i\beta_n$ or $\rho=\alpha_i\beta_n\cdots\beta_1\alpha_i$. Thus
		$$\psi(\beta_n\cdots\beta_1\alpha_i\beta_n)=(\beta_n\cdots\beta_1,0)(0,(\beta_n\cdots\beta_1)^\vee)(\beta_n,0)=(0,e_{t(\beta_n)}^\vee)(\beta_n,0)=0;$$
		$$\psi(\alpha_i\beta_n\cdots\beta_1\alpha_i)=(0,(\beta_n\cdots\beta_1)^\vee)(\beta_n\cdots\beta_1,0)(0,(\beta_n\cdots\beta_1)^\vee)=(0,e_{t(\beta_n)}^\vee)(0,(\beta_n\cdots\beta_1)^\vee)=0.$$
		
		Let $\rho=\alpha_{n}\cdots\alpha_{1}-\beta_m\cdots\beta_1$ is a relation fitting the condition (R1'). Then there exist a nonzero path $p$ in $A_E$, such that $p\alpha_{n}\cdots\alpha_{1}$ and $p\beta_m\cdots\beta_1$ are special paths in $A_E$. By the proof in Proposition \ref{EA is fsbc}, there are arrows $\alpha_i$ and $\beta_j$ induced by maximal paths $m_1$ and $m_2$ in $A$ which are subarrows of $\alpha_{n}\cdots\alpha_{1}$ and $\beta_m\cdots\beta_1$ respectively. Moreover, $m_1=\alpha_{i-1}\cdots\alpha_{1} p\alpha_{n}\cdots\alpha_{i+1}$ and $m_2=\beta_{j-1}\cdots\beta_{1} p\beta_{m}\cdots\beta_{j+1}$. Therefore, 
		$$\psi(\rho)=(0,p^\vee)-(0,p^\vee)=0.$$
		
		In conclusion, the surjection $\psi$ induces a surjection of $k$-algebras from $A_E$ to $T(A)$. Moreover, since $\dim_k A_E=\dim_k T(A)$, $A_E$ is isomorphic to $T(A)$.
	\end{proof}
	
	In fact, Theorem~\ref{Trivial extension of mono} generalizes \cite[Theorem~1.2]{Sch} and \cite[Theorem~4.3]{GS}. In these two papers, the authors showed that the trivial extensions of gentle algebras are Brauer graph algebras, and that the trivial extensions of almost gentle algebras are Brauer configuration algebras. In particular, gentle and almost gentle algebras are monomial algebras, while Brauer graph algebras and Brauer configuration algebras are symmetric $f_s$-BCAs. 

	We now present several examples of trivial extensions of monomial algebras. These monomial algebras are far from being gentle or almost gentle, as their defining relations are not quadratic.
	
	\begin{example}
		Consider $A=kQ/I$, where $Q$ is given by the following quiver
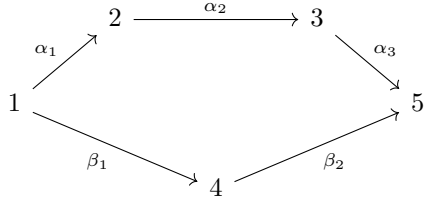
\[\begin{tikzcd}
	& 2 && 3 \\
	\\
	1 &&&& 4
	\arrow["{\alpha_2}", from=1-2, to=1-4]
	\arrow["{\alpha_3}", from=1-4, to=3-5]
	\arrow["{\alpha_1}", from=3-1, to=1-2]
	\arrow["\beta"', from=3-1, to=3-5]
\end{tikzcd}\]
		and $I=\langle \alpha_3\alpha_2\alpha_1\rangle$.
		
		Then we have $\mathcal{M}=\{p_1:=\alpha_2\alpha_1,p_2:=\alpha_3\alpha_2,\beta\}$. Moreover, consider the $f_s$-BC $E_A$ of $A$:
		$$E=\{(e_1,p_1),(e_2,p_1),(e_3,p_1),(e_2,p_2),(e_3,p_2),(e_4,p_2),(e_1,\beta),(e_4,\beta)\};$$
		$$P((e_1,p_1))=\{(e_1,p_1),(e_1,\beta)\};$$
		$$P((e_2,p_1))=\{(e_2,p_1),(e_2,p_2)\};$$
		$$P((e_3,p_1))=\{(e_3,p_1),(e_3,p_2)\};$$
		$$P((e_4,p_2))=\{(e_4,p_2),(e_4,\beta)\};$$
		$$L((e_2,p_1))=\{(e_2,p_1),(e_2,p_2)\},$$
		and $L(e)=\{e\}$, otherwise. The degree function $d$ on $E_A$ is defined as follows:
		\begin{itemize}
			\item $d((e_1,p_1))=d((e_2,p_1))=d((e_3,p_1))=l(p_1)+1=3$;
			\item $d((e_2,p_2))=d((e_3,p_2))=d((e_4,p_2))=l(p_2)+1=3$;
			\item $d((e_1,\beta))=d((e_4,\beta))=l(\beta)+1=2$.
		\end{itemize}
		Thus the quiver correspond to $E_A$ is given by:
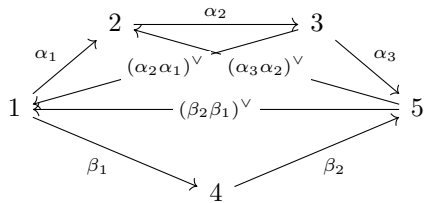
\[\begin{tikzcd}
	& 2 && 3 \\
	\\
	1 &&&& 4
	\arrow["{\alpha_2}", from=1-2, to=1-4]
	\arrow["{(\alpha_2\alpha_1)^{\vee}}"{description, pos=0.6}, from=1-4, to=3-1]
	\arrow["{\alpha_3}", from=1-4, to=3-5]
	\arrow["{\alpha_1}", from=3-1, to=1-2]
	\arrow["\beta"', from=3-1, to=3-5]
	\arrow["{(\alpha_3\alpha_2)^{\vee}}"{description, pos=0.4}, from=3-5, to=1-2]
	\arrow["{\beta^{\vee}}", shift left=3, curve={height=-12pt}, from=3-5, to=3-1]
\end{tikzcd}\]
		It is easy to check that $\dim_k A_E=2\dim_k A=20$ and $T(A)\cong A_E$.
		
	\end{example}
	
		\begin{example}\label{ex:tri-ext-quotient-x^3}
		Consider $A=k[x]/\langle x^3\rangle$. The unique maximal path in $A$ is given by $x^2$. Let $x^2=(e_1\rightarrow e_2\rightarrow e_3)$ with $e_1=e_2=e_3=1$ in $A$. Therefore, consider the $f_s$-BC $E_A$ of $A$:
		$$E=\{(e_1,x^2),(e_2,x^2),(e_3,x^2)\};$$
		$$P((e_1,x^2))=E;$$
		$$L((e_1,x^2))=\{(e_1,x^2),(e_2,x^2)\};$$
		$$L((e_3,x^2))=\{(e_3,x^2)\}.$$
	 	The degree function on $E_A$ is equal to $3$ for all angles in $E$.
		Thus the quiver correspond to $E_A$ is given by:
		$$
		\begin{tikzcd}
			\bullet \arrow["(x^2)^\vee=:y"', loop, distance=2em, in=215, out=145] \arrow["x"', loop, distance=2em, in=35, out=325]
		\end{tikzcd}$$
		Actually, $A_E\cong k[x,y]/\langle x^3,y^2\rangle$.	This is the $f_s$-BCA in Example \ref{ex-alg:2-kx/x^3}.
		It is easy to check that $\dim_k A_E=2\dim_k A=6$ and $T(A)\cong A_E$.
		
	\end{example}
	
	It was shown in \cite{PS, Rin, JS} that $A$ is gentle if and only if $T(A)$ is special biserial. In this case, $T(A)$ is a Brauer graph algebra by \cite[Theorem~1.1]{Sch}. In \cite[Question~4.5]{GS}, Green and Schroll posed the open question of whether an algebra $A$ must be almost gentle if $T(A)$ is a Brauer configuration algebra---a question that remains unsolved. 

	In our setting, however, we provide an example showing that $T(A)$ being a symmetric \fsbca does not necessarily imply that $A$ is a monomial algebra, as demonstrated below.

\begin{example}\label{special cut}
		Consider $A=kQ_A/I_A$ and $B=Q_B/I_B$ with $Q_A$ and $Q_B$ are given by the following quivers:
		
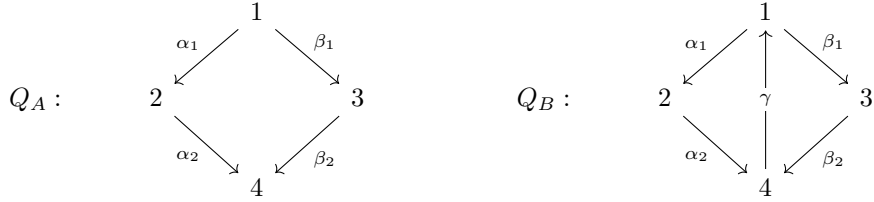
\[\begin{tikzcd}
			&& 1 &&&&& 1 \\
			{Q_A: } & 2 && 3 && {Q_B:} & 2 && 3 \\
			&& 4 &&&&& 4
			\arrow["{\alpha_1}"', from=1-3, to=2-2]
			\arrow["{\beta_1}", from=1-3, to=2-4]
			\arrow["{\alpha_1}"', from=1-8, to=2-7]
			\arrow["{\beta_1}", from=1-8, to=2-9]
			\arrow["{\alpha_2}"', from=2-2, to=3-3]
			\arrow["{\beta_2}", from=2-4, to=3-3]
			\arrow["{\alpha_2}"', from=2-7, to=3-8]
			\arrow["{\beta_2}", from=2-9, to=3-8]
			\arrow["\gamma"{description}, from=3-8, to=1-8]
		\end{tikzcd}\]
		and $I_A=\langle \alpha_2\alpha_{1}-\beta_2\beta_{1}\rangle$, $I_B=\langle \alpha_2\alpha_{1}-\beta_2\beta_{1}, \beta_{1}\gamma\alpha_2, \alpha_{1}\gamma\beta_2, \gamma\alpha_2\alpha_{1}\gamma\rangle$. By \cite[Theorem 3.9]{FP}, $B$ is the trivial extension of $A$. At the same time, $B$ is a symmetric $f_s$-BCA as shown in Example \ref{ex-alg:1-LBGA}.		
	\end{example}
	
		It was shown in \cite[Corollary~7.15]{LL} that the class of finite-dimensional, representation-finite $f_s$-BCAs is closed under derived equivalence. The result in this section provides a potential counterexample in the representation-infinite case.

	\begin{example}
		\begin{tikzcd}
			&& \bullet &&&&& \bullet \\
			{Q:} & \bullet & \bullet & \bullet && {Q_{T(A)}:} & \bullet & \bullet & \bullet \\
			&& \bullet &&&&& \bullet
			\arrow["{\alpha_1}"', from=1-3, to=2-2]
			\arrow["{\gamma_1}"', from=1-3, to=2-3]
			\arrow["{\epsilon_1}", from=1-3, to=2-4]
			\arrow["{\alpha_1}"', from=1-8, to=2-7]
			\arrow["{\gamma_1}"', from=1-8, to=2-8]
			\arrow["{\epsilon_1}", from=1-8, to=2-9]
			\arrow["{\alpha_2}"', from=2-2, to=3-3]
			\arrow["{\gamma_2}"', from=2-3, to=3-3]
			\arrow["{\epsilon_2}", from=2-4, to=3-3]
			\arrow["{\alpha_2}"', from=2-7, to=3-8]
			\arrow["{\gamma_2}"', from=2-8, to=3-8]
			\arrow["{\epsilon_2}", from=2-9, to=3-8]
			\arrow["{\beta_2}"', curve={height=18pt}, from=3-8, to=1-8]
			\arrow["{\beta_1}", curve={height=-18pt}, from=3-8, to=1-8]
		\end{tikzcd}

\noindent Consider the canonical algebra $A$ of type $(2,2,2)$ with the quiver given by $Q$. To be more specific, $A=kQ/\langle\alpha_2\alpha_{1}+\gamma_2\gamma_1+\epsilon_2\epsilon_1\rangle$. Then the quiver of the trivial extension of $A$ can be given by $Q_{T(A)}$ which has a specific description in \cite[Example 2.5]{FSTTV}. Actually, by using \cite[Theorem 1.1]{FSTTV}, $T(A)=kQ_{T(A)}/I$ with $I$ generated by following relations:
		
		\begin{itemize}
			\item $\alpha_2\alpha_{1}+\gamma_2\gamma_1+\epsilon_2\epsilon_1$;
			
			\item $\beta_2\gamma_2$, $\gamma_1\beta_2$, $\beta_{1}\epsilon_2$, $\epsilon_1\beta_{1}$, $\gamma_1\beta_1\alpha_2$, $\alpha_{1}\beta_1\gamma_2$, $\epsilon_1\beta_2\gamma_2$, $\alpha_{1}\beta_2\epsilon_2$;
			
			\item $\gamma_1\beta_{1}\gamma_2\gamma_1$, $\gamma_2\gamma_1\beta_{1}\gamma_2$, $\beta_{1}\gamma_2\gamma_1\beta_{1}$,  $\epsilon_1\beta_2\epsilon_2\epsilon_1$, $\epsilon_2\epsilon_1\beta_2\epsilon_2$, $\beta_2\epsilon_2\epsilon_1\beta_2$;
			
			\item $\alpha_{1}\beta_{i}\alpha_2\alpha_{1}$, $\alpha_2\alpha_{1}\beta_{i}\alpha_2$, $\beta_{i}\alpha_2\alpha_{1}\beta_{i}$, with $i=1,2$;
			
			\item $\beta_{1}\gamma_2\gamma_1-\beta_2\epsilon_2\epsilon_1$, $\gamma_2\gamma_1\beta_{1}-\epsilon_2\epsilon_1\beta_2$, $\beta_{1}\gamma_2\gamma_1+\beta_2\alpha_2\alpha_{1}$, $\gamma_2\gamma_1\beta_{1}+\alpha_2\alpha_{1}\beta_2$, $\beta_{1}\alpha_2-\beta_2\alpha_2$, $\alpha_{1}\beta_{1}-\alpha_{1}\beta_2$.
		\end{itemize}

		It is well-known that $A$ is derived equivalent to the path algebra $B$ of type $\widetilde{D_4}$; see, for example, in \cite[Theorem~3.5]{Le}. Then $T(A)$ and $T(B)$ are also derived equivalent by \cite[Theorem~3.1]{Ric1}. By Theorem~\ref{Trivial extension of mono}, $T(B)$ is a symmetric $f_s$-BCA. It would be interesting to determine whether the above algebra $T(A)$ is isomorphic to some $f_s$-BCA. However, this seems unlikely, since the defining ideal $I$ contains the relation $\alpha_2\alpha_{1}+\gamma_2\gamma_1+\epsilon_2\epsilon_1$ together with several anti-commutative relations.
	\end{example}

	\section{Admissible cuts on symmetric $f_s$-BCAs}

	One way to construct new algebras by deleting arrows from quivers is through the notion of admissible cuts of finite-dimensional algebras, a concept studied for instance in \cite{EF, FP2, GS, Sch}. The aim of this section is to generalize the main result of \cite{GS} from Brauer configuration algebras to symmetric $f_s$-BCAs.

	Let $\Lambda = kQ_\Lambda / I_\Lambda$ be a symmetric $f_s$-BCA with a free f-degree function. Assume that the associated angle set $E$ is finite, which is equivalent to $\Lambda$ being finite-dimensional. Without loss of generality, we may suppose that $I_\Lambda$ is admissible, since \cite[Definition~6.1]{LL} provides a construction of presentations of $\Lambda$ with admissible ideals. 

	All special paths (that is, the paths corresponding to $\lseq{e}$ in the $f_s$-BC $E$) are cycles in $Q_\Lambda$, since the associated $f_s$-BC is symmetric. We refer to these cycles as {\it special cycles}. This notion corresponds to the special cycles in Brauer configuration algebras~\cite{GS2}. Denote by $\mathcal{S}$ the set of all special cycles, considered up to cyclic permutation.

	\begin{definition}\textnormal{(Compare~to~\cite[Definition 3.2]{FP2}~and~\cite[Definition 5.1]{GS})} Let $\Lambda = kQ_\Lambda / I_\Lambda$ be a finite-dimensional symmetric $f_s$-BCA with a free f-degree function, and let $\mathcal{S}=\{C_1, \dots, C_t\}$ be a set of representatives of the equivalence classes of special cycles under cyclic permutation.
		\begin{enumerate}
			\item A cutting set $D$ of $Q_\Lambda$ is a subset of arrows in $Q_\Lambda$ containing exactly one arrow from each special cycle corresponding to a representative $C_i$ of the equivalence classes, for $i = 1, \dots, t$. 
			\item We define the cut algebra associated with $D$ as 
\[
kQ_\Lambda / \langle I_\Lambda \cup D \rangle,
\]
where $\langle I_\Lambda \cup D \rangle$ denotes the ideal generated by $I_\Lambda \cup D$.
\item Furthermore, we call a cutting set $D$ an admissible cut if it satisfies the following conditions:
\begin{enumerate}[(c1)]
    \item $D$ consists of $t$ distinct arrows in $Q_\Lambda$;
    \item each arrow in $D$ appears in any special cycle $C_i$ at most once.
\end{enumerate}
		\end{enumerate}
	\end{definition}

	\begin{remark}
		In fact, this definition is a generalization of the notion of an admissible cut given in \cite[Section~4]{Sch} for Brauer graph algebras and in \cite[Definition~5.1]{GS2} for Brauer configuration algebras. Note that our definition of a cutting set actually coincides with their notion of an admissible cut, since for Brauer graph algebras and Brauer configuration algebras a cutting set automatically satisfies conditions~(c1) and~(c2).
The following examples illustrate that these conditions are indeed necessary:

(1) \textbf{Condition (c1) is necessary.} 
Consider the $f_s$-BCA $B$ in Example~\ref{special cut}, which has two distinct special cycles $\{\gamma\alpha_2\alpha_1,\gamma\beta_2\beta_1\}$ up to cyclic permutation. 
If we take the cutting set $D = \{\gamma\}$ consisting of only one arrow, the corresponding cut algebra is the algebra $A$ in Example~\ref{special cut}. 
This algebra is not monomial, and thus lies outside the scope of our discussion.

(2) \textbf{Condition (c2) is necessary.} 
In this section, we aim to establish a one-to-one correspondence (see Corollary~\ref{1-1 cor}) between monomial algebras 
and the cut algebras of symmetric $f_s$-BCAs with free f-degree functions.
However, some $f_s$-BCAs with free f-degree functions cannot be realized as the trivial extension of their cut algebras. For instance, consider the $f_s$-BCA $$A\cong k\langle x,y\rangle/\langle x^2y-yx^2,xy^2-y^2x,xyx,yxy,x^3,y^3\rangle$$ in Example~\ref{ex-alg:3-fBCA-without-admissible-cut}. 
For any choice of a single arrow as a cutting set, the corresponding cut algebra is isomorphic to the algebra $B=k[x]/\langle x^3\rangle$.
From Example~\ref{ex:tri-ext-quotient-x^3} we know that the trivial extension of $B$ is the $f_s$-BCA in Example~\ref{ex-alg:2-kx/x^3}, which has a different dimension from $A$. In fact, by condition~(c2), this $f_s$-BCA $A$ admits no admissible cut, and therefore does not appear in the correspondence established in Corollary~\ref{1-1 cor}.
	\end{remark}
	
	We now show that condition (c1) ensures that the cut algebra is a monomial algebra.
	
	\begin{proposition}\label{prop:cut-monomial}
		Let $\Lambda = kQ_\Lambda / I_\Lambda$ be a symmetric $f_s$-BCA with a free f-degree function, and let $D$ be a cutting set of $Q_\Lambda$ satisfying condition~(c1). Define a quiver $Q$ by setting $Q_0 = (Q_\Lambda)_0$ and $Q_1 = (Q_\Lambda)_1 \setminus D$. Then the cut algebra $kQ_\Lambda / \langle I_\Lambda \cup D \rangle$ associated with $D$ is isomorphic to $kQ / \langle I_\Lambda \cap kQ \rangle$. Moreover, the algebra $kQ / \langle I_\Lambda \cap kQ \rangle$ is monomial.
	\end{proposition}
	
	\begin{proof}
		The inclusion of quivers $Q \subset Q_\Lambda$ induces a $k$-algebra homomorphism  
\[
f : kQ \longrightarrow kQ_\Lambda / \langle I_\Lambda \cup D \rangle.
\]
The map $f$ is surjective, since every nonzero element in $kQ_\Lambda / \langle I_\Lambda \cup D \rangle$ is a linear combination of paths that do not contain any arrow from $D$. Hence, for each such path $p'$ in $Q_\Lambda$, there exists a corresponding path $p$ in $Q$ such that $f(p) = p'$. By the First Isomorphism Theorem, we obtain
\[
kQ_\Lambda / \langle I_\Lambda \cup D \rangle \cong kQ / \langle I_\Lambda \cap kQ \rangle.
\]
Furthermore, the algebra $kQ / \langle I_\Lambda \cap kQ \rangle$ is monomial, since every path appearing in a relation of type~1 in $I_\Lambda$ contains an arrow from $D$; equivalently, no such relation lies entirely in $kQ$.
	\end{proof}
	
	The following result shows that if we start with a monomial algebra and take an appropriate admissible cut in its trivial extension, then the original monomial algebra is isomorphic to the corresponding cut algebra.

	\begin{theorem}
		Let $A = kQ / I$ be a monomial algebra with the set of maximal paths $\mathcal{M}$, and let $T(A) = kQ_{T(A)} / I_{T(A)}$ be its trivial extension. Denote by $D = \{\alpha_m \mid m \in \mathcal{M}\}$ the set of new arrows in $Q_{T(A)}$. Then $D$ is an admissible cut of $Q_{T(A)}$, and the cut algebra associated with $D$ is isomorphic to $A$.

	\end{theorem}
	
	\begin{proof}
		By Theorem \ref{Trivial extension of mono}, we can view $T(A)$ as an $f_s$-BCA with $\mathcal{S} = \{C_1, \ldots, C_t\}$. Moreover, each special cycle in $T(A)$ corresponds uniquely to a maximal path in $\mathcal{M}$. From the construction of $T(A)$, we see that each special cycle $C_i$ (for $i = 1, \ldots, t$) contains exactly one arrow from the set $D$.  Hence, $D$ forms an admissible cut of $T(A)$. Furthermore, by the proof of Theorem \ref{Trivial extension of mono}, the cut algebra associated with $D$ is isomorphic to $A$.
	\end{proof}
	
	The next result shows that if one starts with a symmetric $f_s$-BCA with a free f-degree function and an admissible cut $D$ on its quiver, then the trivial extension of the cut algebra associated with $D$ is isomorphic to the original symmetric $f_s$-BCA.

	\begin{theorem}
		Let $\Lambda = kQ_\Lambda / I_\Lambda$ be a symmetric $f_s$-BCA with a free f-degree function, and let $D$ be an admissible cut of $Q_\Lambda$. Denote by $A = kQ / I$ the cut algebra associated with $D$. Then the trivial extension $T(A)$ is isomorphic to $\Lambda$.
	\end{theorem}
	
	\begin{proof}
		By Proposition \ref{prop:cut-monomial}, it is clear that $I$ is generated by paths. In fact, the special cycles in $Q_\Lambda$ are of the form $C = p\alpha q$ for some $\alpha \in D$. Since $\Lambda$ is symmetric and $C$ is a special cycle, both $qp\alpha$ and $\alpha qp$ are also special cycles. Hence $qp \notin I_\Lambda$, and consequently $qp \notin I_\Lambda \cap kQ$. 
Because $I_\Lambda$ is admissible and $\Lambda$ is an $f_s$-BCA, if there exist arrows $\beta$ and $\gamma$ in $Q_\Lambda$ such that $\beta qp$ and $qp\gamma$ are nonzero in $\Lambda$, then necessarily $\beta = \gamma = \alpha$. This implies that for every arrow $\beta'$ in $Q$, we have $\beta' qp = qp \beta' = 0$ in $A$. Therefore, $qp$ is a maximal path in $A$. 
By the definition of the $f_s$-BCA associated with $A$ and by Theorem~\ref{Trivial extension of mono}, we conclude that $T(A) \cong \Lambda$.
	\end{proof}

	Consider the set of pairs $(\Lambda, D)$ where $\Lambda = kQ_\Lambda / I_\Lambda$ is a symmetric $f_s$-BCA with a free f-degree function and $D$ is an admissible cut of $Q_\Lambda$. We say that two pairs $(\Lambda, D)$ and $(\Lambda', D')$ are equivalent if there exists a $k$-algebra isomorphism from $\Lambda$ to $\Lambda'$ that sends $D$ to $D'$. Denote by $\mathcal{Y}$ the set of equivalence classes under this relation. It is clear that all cut algebras corresponding to pairs in the same equivalence class are isomorphic. By combining the previous two theorems, we obtain the following main result of this section.
	
	\begin{corollary}\label{1-1 cor}
		There is a bijection
\[
\phi : \mathcal{A} \longrightarrow \mathcal{Y}
\]
between the set $\mathcal{A}$ of isomorphism classes of monomial algebras 
and the set $\mathcal{Y}$ of equivalence classes of pairs consisting of a symmetric $f_s$-BCA with a free f-degree function and an admissible cut, as defined above. 
For $A \in \mathcal{A}$, the map is given by
\[
\phi(A) = (T(A), D),
\]
where $D = \{ \alpha_m \mid \text{$m$ is a maximal path in $A$} \}$. 
Conversely, for a pair $(\Lambda, D)$, we have
\[
\phi^{-1}((\Lambda, D)) = A,
\]
where $A$ denotes the isomorphism class of the cut algebra associated with the admissible cut $D$.
	\end{corollary}

	We conclude this section with the following example.

	\begin{example}\label{ex:skew-BGA} (Example \ref{ex-alg:1-LBGA} revisited)
		Let $\Lambda=Q_\Lambda/I_\Lambda$ be the $f_s$-BCA in Example \ref{ex-alg:1-LBGA}, where $Q_\Lambda$ is given by 
		\[\begin{tikzcd}
	& {1} \\
	{2} && {3} \\
	& {4}
	\arrow["{\alpha_1}"', from=1-2, to=2-1]
	\arrow["{\beta_1}", from=1-2, to=2-3]
	\arrow["{\alpha_2}"', from=2-1, to=3-2]
	\arrow["{\beta_2}", from=2-3, to=3-2]
	\arrow["{\gamma}"{description}, from=3-2, to=1-2]
\end{tikzcd}\]
and $I_B=\langle \alpha_2\alpha_{1}-\beta_2\beta_{1}, \beta_{1}\gamma\alpha_2, \alpha_{1}\gamma\beta_2, \gamma\alpha_2\alpha_{1}\gamma\rangle$.

Consider $(\Lambda,\{\alpha_2,\beta_1\})\in\mathcal{Y}$. Then the cut algebra is given by $A_1=kQ_1/I_1$, where $Q_1$ is given by 
\[\begin{tikzcd}
	& {1} \\
	{2} && {3} \\
	& {4}
	\arrow["{\alpha_1}"', from=1-2, to=2-1]
	\arrow["{\beta_2}", from=2-3, to=3-2]
	\arrow["{\gamma}"{description}, from=3-2, to=1-2]
\end{tikzcd}\]
and $I_1=\langle\alpha_1\gamma\beta_2\rangle$. In fact, $\Lambda\cong T(A_1)$.

Consider $(\Lambda,\{\alpha_2,\beta_2\})\in\mathcal{Y}$. Then the cut algebra is given by $A_2=kQ_2$, where $Q_2$ is given by 
\[\begin{tikzcd}
	& {1} \\
	{2} && {3} \\
	& {4}
	\arrow["{\alpha_1}"', from=1-2, to=2-1]
	\arrow["{\beta_1}", from=1-2, to=2-3]
	\arrow["{\gamma}"{description}, from=3-2, to=1-2]
\end{tikzcd}\]
In fact, $\Lambda \cong T(A_2)$. Since $A_2$ is the path algebra of type $D_4$, it is also a skew-gentle algebra (see for example in \cite{GdlP}). Hence, $\Lambda$ is actually a skew-Brauer graph algebra in the sense of \cite{EGV, Soto}.

	\end{example}
	
	\section*{Declarations}
	
	\subsection*{Ethical Approval}
	This declaration is not applicable.
	\subsection*{Funding}
	This research is supported by NSFC (No.12031014)  and China Scholarship Council (No. 202506040127).
	\subsection*{Availability of data and materials}
	The datasets generated during the current study are available from the corresponding author on reasonable request.
	\subsection*{Acknowledgments}
		We would like to thank the referee for careful reading and helpful suggestions to improve our presentations.
	
	{}
\end{document}